\documentclass[a4paper,10pt,leqno]{amsart}
\usepackage{amsmath,amsfonts,amsthm,amssymb}
\usepackage[alphabetic]{amsrefs}
\usepackage{enumerate}
\usepackage{mathrsfs}
\usepackage{enumerate, xspace}
\usepackage[pdftex]{graphicx}
\usepackage{color}

\long\def\symbolfootnote[#1]#2{\begingroup
\def\thefootnote{\fnsymbol{footnote}}\footnote[#1]{#2}\endgroup}

\newtheorem{theorem}{Theorem}[section]
\newtheorem{lem}[theorem]{Lemma}
\newtheorem{thm}[theorem]{Theorem}
\newtheorem{sublemma}[theorem]{Sublemma}
\newtheorem{prop}[theorem]{Proposition}
\newtheorem{cor}[theorem]{Corollary}

\newtheorem{claim}{Claim}

\theoremstyle{definition}
\newtheorem{rem}[theorem]{Remark}
\newtheorem{defin}[theorem]{Definition}
\newtheorem{constr}[theorem]{Construction}
\newtheorem{ex}[theorem]{Example}

\newtheorem{step}{Step}

\newcommand{\R}{\mathbb{R}}
\newcommand{\Z}{\mathbb{Z}}

\begin{document}

\title{Balanced walls for random groups}

\author{John~M.~Mackay}
\address{University of Bristol,
School of Mathematics, University Walk, Bristol, BS8 1TW, UK}
\email{john.mackay@bristol.ac.uk}

\author{Piotr Przytycki}
\address{Dep.\ of Math.\ \& Stat., McGill University\\
Burnside Hall, Room 1005, 805 Sherbrooke St.\ W\\
Montreal, Quebec, Canada H3A 2K6\\
and Inst.\ of Math., Polish Academy of Sciences\\
 \'Sniadeckich 8, 00-656 Warsaw, Poland}
\email{piotr.przytycki@mcgill.ca}

\begin{abstract}
We study a random group $G$ in the Gromov density model and its
Cayley complex $X$. For density $<\frac{5}{24}$ we define walls
in $X$ that give rise to a nontrivial action of $G$ on a CAT(0)
cube complex. This extends a result of Ollivier and Wise, whose
walls could be used only for density $<\frac{1}{5}$. The
strategy employed might be potentially extended in future to
all densities $<\frac{1}{4}$.
\end{abstract}

\maketitle

\section{Introduction}
\label{sec:intro}

Following Gromov \cite{Gro} and Ollivier \cite{O-sur}, we study random
groups in the following \emph{Gromov density model}. Fix $m$ letters $S
= \{s_1, s_2, \ldots, s_{m}\}$, and let $S^{-1}$ denote the formal inverses of $S$. Choose a \emph{density} $d \in
(0,1)$. A \emph{random group (presentation) at density $d$ and
length $l$} is a group $G=\langle S | R\rangle$, where $R$ is a
collection of $\lfloor(2m-1)^{dl}\rfloor$ cyclically reduced
words in $S\cup S^{-1}$ of length $l$ chosen independently and
uniformly at random. In our article we assume additionally that $l$ is even.
A random group (presentation) at density
$d$ has property P \emph{with overwhelming probability} (shortly \emph{w.o.p.}) if the probability of $G$ having P tends to $1$ as $l\rightarrow\infty$.

Gromov and Ollivier proved that for $d>\frac{1}{2}$ a random
group $G$ is w.o.p. $\Z/2\Z$ (we assumed $l$ to be even), while
for $d<\frac{1}{2}$ it is w.o.p.\ non-elementary hyperbolic
with hyperbolicity constant linear in $l$, torsion free, and
with contractible Cayley complex~$X$ \cite{Gro,O-hyp}. For
$d>\frac{1}{3}$ a random group $G$ has w.o.p.\ Kazhdan's
property (T), which was proved by \.Zuk \cite{Z} (and completed
by Kotowski--Kotowski \cite{KK}). On the other hand, property
(T) fails for $d<\frac{1}{5}$, since in that range
Ollivier--Wise proved that w.o.p.\ $G$ acts nontrivially on a
CAT(0) cube complex (they also proved that the action is proper
for $d<\frac{1}{6}$) \cite{OW}.

Their cube complex is obtained from Sageev's construction
\cite{S}, using an action of $G$ on a suitable space with
walls. They use the following wall structure in the Cayley
complex~$X$ of $G=\langle S | R\rangle$: Consider the graph
whose vertices are edge midpoints of $X$ and whose edges are
pairs of opposite edge midpoints in the $2$--cells of $X$.
\emph{Hypergraphs} are connected components of that graph,
immersed in $X$ in such a way that its edges are mapped to the
diagonals of the $2$--cells. Ollivier--Wise prove that for
$d<\frac{1}{5}$ a hypergraph is w.o.p.\ an embedded tree,
separating $X$ essentially, with cocompact stabilizer $H$.
Thus, possibly after replacing $H$ with its index $2$ subgroup
preserving the halfspaces, the number of relative ends
satisfies $e(G,H)>1$, and hence the action of $G$ on the CAT(0)
cube complex given by Sageev's construction is nontrivial.
However, for $d>\frac{1}{5}$, w.o.p.\ hypergraphs
self-intersect, thus we do not have control on $H$.

The aim of our paper is to introduce a new wall structure, by
replacing the antipodal relation inside $2$--cells by a
different relation, so that the  resulting hypergraphs are
embedded trees and we can perform Sageev's construction. While
our strategy is designed to work up to density $\frac{1}{4}$,
the technical complications that arise force us for the moment
to content ourselves with the following.

\begin{thm}
\label{thm:main} In the Gromov density model at density
$<\frac{5}{24}$, a random group w.o.p.\ acts nontrivially on a
$\mathrm{CAT(0)}$ cube complex and does not satisfy Kazhdan's
property $\mathrm{(T)}$.
\end{thm}

The CAT(0) cube complex in Theorem~\ref{thm:main} can be chosen
to be finite-dimensional and cocompact, see
Remark~\ref{rem:finite dim}.

We assumed $l$ to be even only to have an easy proof of
Lemma~\ref{lem:transitive}, which would have been otherwise
slightly more difficult and would also require $d>\frac{1}{8}$
(to which we actually could have restricted). For $l$ odd one
subdivides the edges of $X$ into two and replaces $l$
with~$2l$.

\medskip

\noindent \textbf{Strategy outline.} In the remaining part of
the Introduction we outline our strategy for the proof of
Theorem~\ref{thm:main}. Our starting point is
Figure~\ref{fig:colliding-walls}, left, (Figure~21 from
\cite{OW}) explaining why at $d>\frac{1}{5}$ a hypergraph
obtained from the antipodal relation in $2$--cells
self-intersects. The principal reason for the self-intersection
to appear is the sharp turn that the hypergraph makes in the
union $T$ of the two $2$--cells $C,C'$ on the left. As $d$
approaches $\frac{1}{4}$, the possible length $|A|$ of the
common path $A=C\cap C'$ approaches $\frac{1}{2}l$. Hence the
distance in $T^{(1)}$ between the endpoints of a hypergraph
segment tends to $0$, see Figure~\ref{fig:colliding-walls},
right. This is bad, since can easily be turned into a
self-intersection by adding a third $2$--cell to $T$ as in
Figure~\ref{fig:colliding-walls}, left.
\begin{figure}
	\centering
	\def\svgwidth{0.9\columnwidth}
	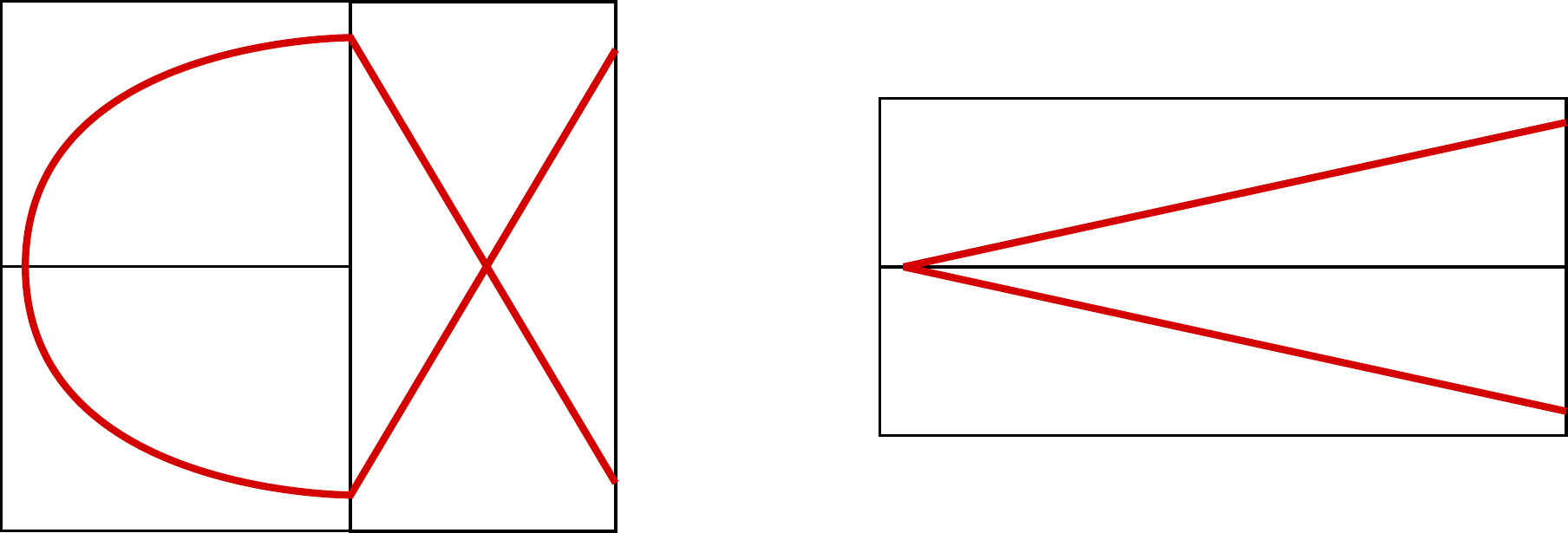
	\caption{}
	\label{fig:colliding-walls}
\end{figure}

To remedy this, whenever $|A|>\frac{1}{4}l$ we replace the
antipodal relation in one of the two $2$--cells of $T$, say in
$C'$, by the relation $\sim$ described in Figure~2.
Specifically, we consider two subpaths $\alpha_+,\alpha_-$ of
$A$ of length $\lceil|A|-\frac{1}{4}l\rceil$ containing the endpoints of
$A$. Let $s_\pm$ be the symmetry of $\alpha_\pm$ exchanging its
endpoints. If $x,y$ are antipodal edge midpoints of $C'$ and
$y$ lies in the interior of $\alpha_\pm$, then we put $x\sim
s_\pm(y)$, otherwise let $x\sim y$.
\begin{figure}
	\centering
	\def\svgwidth{0.4\columnwidth}
	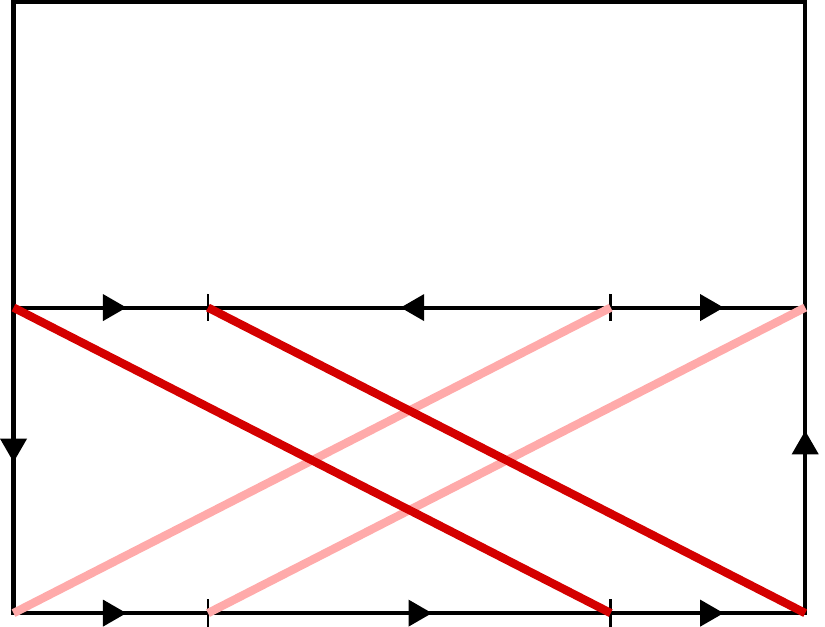
	\caption{}
	\label{fig:glue-mix}
\end{figure}

This relation has the following advantage over the antipodal
one. Let $x,x'$ be in the same hypergraph of $T$. We claim that
the distance between $x$ and $x'$ in $T^{(1)}$ is bounded below
by
$$\mathrm{Bal}(T)=\frac{1}{4}(|T|+1)l -|A|,$$
where $|T|=2$ is the number of $2$--cells in $T$. This value
will be called the \emph{balance} of $T$, and is bounded from
above by $\frac12 l$.  Notice that for $d<\frac{1}{4}$ it is
also bounded below by $\frac{1}{4}l$, since $|A|<\frac{1}{2}l$.

To justify the claim, there are four cases to consider. If
$x,x'$ are antipodal in $C$ or $C'$, there is nothing to prove.
If $x,x'$ are both in $C'$ and $x'=s_\pm(y)$, where $y$ is
antipodal to $x$, then since the distance $|x',y|$ satisfies
$|x',y|<|A|-\frac{1}{4}l$, it suffices to use the triangle
inequality. Otherwise $x\in C-C',\ x'\in C'-C$, and the
hypergraph segment $xx'$ crosses $A$ in an edge midpoint $y$
such that: either $y$ is antipodal to both $x$ and $x'$ and
lies outside of the interiors of $\alpha_\pm$; or $y$ lies in
the interior of, say, $\alpha_+$ and is antipodal to $x$ in $C$
while $s_+(y)$ is antipodal to $x'$ in $C'$. In the first
situation $y$ is at distance $\leq\frac{1}{4}l$ from the
endpoints of $A$, so that the distance between $x,x'$ is $\geq
\frac{1}{2}l+\frac{1}{2}l-2(\frac{1}{4}l)$, as desired. In the
second situation the sum of the distances of $y,s_+(y)$ from an
endpoint of $A$ is $\leq
2|A|-(|A|-\frac{1}{4}l)=\frac{1}{4}l+|A|$, so that the distance
between $x,x'$ is $\geq
\frac{1}{2}l+\frac{1}{2}l-(\frac{1}{4}l+|A|)$, which finishes
the proof of the claim.

Note that this is the best estimate we can hope for: consider
the edge midpoint $y\in \alpha_+$ nearly at the endpoint of $A$
and suppose that we attempt to move $x'\sim y$ in $C'$, which
is the antipode of $s_+(y)$. Then either we decrease the
distance between $x'$ and $y$ or we decrease the distance
between $x'$ and the antipode $x$ of $y$ in $C$, and both these
distances were nearly equal $\mathrm{Bal}(T)$.

We call $T$ a \emph{tile} and the relation on the edge
midpoints of $T$ induced by the relations in $C$ and $C'$ a
\emph{balanced tile-wall structure}. We iterate this
construction: whenever two tiles, or a tile and a $2$--cell
have large overlap, we change the antipodal relation into one
that makes the tile-wall structure balanced. The tiles in $X$
will not share $2$--cells, except for very particular
configurations, and will be used instead of $2$--cells in van
Kampen diagrams. One way to think about this is that since we
do not see negative curvature on the original presentation
level, we zoom out and look at tiles instead of $2$--cells,
where we are already able to define walls with negative
curvature behaviour.

There are two technical problems that one would need to
overcome to extend the proof of Theorem~\ref{thm:main} to all
densities $<\frac{1}{4}$. First of all, one needs to understand
the combinatorial complication coming from tiles sharing
$2$--cells (generalisation of assertions (i)---(iii) in
Proposition~\ref{prop:wallconstruction}). Secondly, even for a
tile disjoint from all other ones (as in Step~1 of
Construction~\ref{constr:main}), but glued of two tiles of size
$\geq 3$, we do not not know in general how to define a
balanced tile-wall structure, i.e.\ how to extend Part~1 of
Proposition~\ref{prop:wallconstruction}.

\medskip

\noindent \textbf{Organisation.} In
Section~\ref{sec:isoperimetric} we discuss the isoperimetric
inequality for random groups. In Section~\ref{sec:tiles} we
define tiles, and we equip them with balanced tile-wall
structures in Section~\ref{sec:tile-walls}. We then show that
induced hypergraphs are embedded trees
(Section~\ref{sec:walls}), which are quasi-isometrically
embedded (Section~\ref{sec:quasi-convexity}), and we conclude
with the proof of Theorem~\ref{thm:main}.

\medskip

\noindent \textbf{Acknowledgement.}
We thank the referee for his or her detailed and helpful comments.
The second author was partially supported by the Foundation for Polish Science, National Science Centre DEC-2012/06/A/ST1/00259 and NSERC.

\section{Isoperimetric inequality}
\label{sec:isoperimetric}

In this section we recall Ollivier's isoperimetric inequality
for disc diagrams in random groups, extended to uniformly
bounded non-planar complexes by Odrzyg\'o\'zd\'z.

We always assume that $2$--cells in our complexes are $l$--gons
with $l$ even. A \emph{disc diagram} $D$ is a contractible
$2$--complex with a fixed embedding in $\R^2$. Its
\emph{boundary path} $\partial D$ is the attaching map of the
cell at infinity.

Suppose $Y$ is a $2$--complex, not necessarily a disc diagram.
Let the \emph{size} $|Y|$ denote the number of $2$--cells of
$Y$. If $A$ is a graph, and is not treated as a disc diagram of
size $0$, we denote by $|A|$ the number of $1$--cells in $A$.
The \emph{cancellation} of~$Y$~is
$$\mathrm{Cancel}(Y) = \sum_{e\ \mathrm{an}\ \mathrm{edge}\  \mathrm{of}\ Y} (\deg(e)-1),$$
where $\deg(e)$ is the number of times that $e$ appears as the image of an edge of the attaching map of a $2$--cell of $Y$.
Observe that if $D$ is a disc diagram, $\mathrm{Cancel}(D)$ counts the number of internal edges of $D$.

\begin{rem}
\label{rem:cancel}
Suppose that $Y_i\subset X$ are subcomplexes that are closures of their $2$--cells, and that they do not share $2$--cells. Then
$$\mathrm{Cancel}\big(\bigcup_i Y_i\big)\geq\sum_i\Big(\mathrm{Cancel}(Y_i)+\frac{1}{2}|Y_i\cap \bigcup_{j\neq i}Y_j|\Big).$$
Equality holds if and only if no triple of $Y_i$ shares an edge.
\end{rem}

We say that $Y$ is \emph{fulfilled} by a set of relators $R$ if
there is a combinatorial map from~$Y$ to the presentation
complex $X/G$ that is locally injective around edges (but not
necessarily around vertices). In particular, any subcomplex of
the Cayley complex $X$ is fulfilled by~$R$. Since $X$ is simply
connected, for any closed path $\alpha$ in $X^{(1)}$ there
exists a disc diagram $D$ with a map $D\rightarrow X$ such that
$\partial D$ maps to $\alpha$. Moreover, by cancelling some
$2$--cells we can assume that $D$ is fulfilled by $R$. We say
that $D\rightarrow X$ is a \emph{disc diagram for $\alpha$}.

\begin{thm}[{\cite[Thm 2]{O-iso}}]
\label{thm:isoperimetry}
For each $\varepsilon>0$ w.o.p.\ there is no disc diagram $D$ fulfilling $R$ and satisfying
$$\mathrm{Cancel}(D)>(d+\varepsilon)|D|l.$$
(Equivalently, every disc diagram $D$ fulfilling $R$ has $|\partial D| \geq (1-2d-2\epsilon)|D|l$.)
\end{thm}

We deduce the following:

\begin{lem}
\label{lem:no_short_loop}
Let $d<\frac{1}{4}$. Then w.o.p.\ there is no embedded closed path in $X^{(1)}$ of length $<l$.
\end{lem}
\begin{proof}
Otherwise, let $D\rightarrow X$ be a disc diagram for that closed path.
Case $|D|=1$ is not possible. Otherwise $|D|\geq 2$, and hence
$$\mathrm{Cancel}(D)= \frac{1}{2}(|D|l-|\partial D|)>\frac{1}{2}|D|l-\frac{1}{2}l\geq \frac{1}{4}|D|l,$$
which contradicts Theorem~\ref{thm:isoperimetry}.
\end{proof}

Lemma~\ref{lem:no_short_loop} immediately implies the following corollaries.

\begin{cor}
\label{cor:cells_embed}
Let $d<\frac{1}{4}$. Then w.o.p.\ the boundary paths of all $2$--cells embed in~$X$.
\end{cor}

\begin{cor}
\label{cor:short_geodesic}
Let $d<\frac{1}{4}$. Then w.o.p.\ every path $\alpha$ embedded in $X^{(1)}$ of length $\leq\frac{1}{2}l$ is geodesic in $X^{(1)}$.
\end{cor}

\begin{cor}
\label{cor:no_short_image}
Let $d<\frac{1}{4}$. Then w.o.p.\ there is no immersed closed path $\alpha\colon I\rightarrow X^{(1)}$ with $|\alpha(I)|<l$.
\end{cor}

Another consequence of Theorem~\ref{thm:isoperimetry} is the
following result of Ollivier and Wise, whose proof we include
as a warm-up.

\begin{lem}[{\cite[Cor 1.11]{OW}}]
\label{lem:cells_intersection_connected}
Let $d<\frac{1}{4}$. Then w.o.p.\ for all intersecting $2$--cells $C,C'$ of $X$ we have that $C\cap C'$ is
connected.
\end{lem}

\begin{proof}
If $C\cap C'$ is not connected, then there is in $C\cup C'$ a
homotopically non-trivial embedded closed path $\alpha\cup\alpha'$ of length $\leq l$ with $\alpha$ in $C$ and $\alpha'$ in $C'$. This
contradicts Lemma~\ref{lem:no_short_loop} unless $|\alpha|=|\alpha'|=\frac{1}{2}l$. By Theorem~\ref{thm:isoperimetry}, as in the proof
of Lemma~\ref{lem:no_short_loop}, this shows that
$\alpha\cup\alpha'$ bounds a disc diagram $D$ of size $|D|=1$,
hence consisting of a single $2$--cell $C''$. This contradicts Theorem~\ref{thm:isoperimetry}
with $d<\frac{1}{4}$ for the diagram $C\cup C''$.
\end{proof}

We close with the following variant of Theorem~\ref{thm:isoperimetry} for uniformly bounded non-planar complexes.
We say that a $2$--complex $Y$ is \emph{$(K,K')$--bounded}
if $|Y|\leq K$ and $Y$ is obtained from the disjoint union of its $2$--cells by
gluing them along $\leq K'$ subpaths of their boundary paths.
Note that for $d<\frac{1}{4}$ by Corollary~\ref{cor:cells_embed} and
Lemma~\ref{lem:cells_intersection_connected}, if $Y\subset X$,
then $|Y|\leq K$ implies that $Y$ is
$\big(K,\frac{1}{2}K(K-1)\big)$--bounded.

\begin{prop}[see {\cite[Thm 1.5]{Od}}]
\label{prop:isoperimetry}
For each $K,K',\ \varepsilon>0$ w.o.p.\ there is no $(K,K')$--bounded $2$--complex~$Y$
fulfilling $R$ and satisfying
$$\mathrm{Cancel}(Y)>(d+\varepsilon)|Y|l.$$
\end{prop}

In fact, Odrzyg\'o\'zd\'z proves the following stronger result.
We say that $Y$ has \emph{$L$ fixed paths} if we distinguish $L$ subpaths of the
boundary paths of the $2$--cells in $Y$. We denote their union by $\mathrm{Fix}(Y)$.
A \emph{labelling} of a $2$--complex $Y$ with fixed paths is a combinatorial map from $\mathrm{Fix}(Y)$ to $X^{(1)}/G$.
A \emph{polynomial labelling scheme} is a function assigning to
each $2$--complex $Y$ with fixed paths a set of labellings,
where the cardinality of the set of labellings is bounded by a polynomial in $l$.

\begin{prop}[{\cite[Thm 1.5]{Od}}]
\label{prop:isoperimetry2}
Given a polynomial labelling scheme, for each $K,K',L,\ \varepsilon>0$ w.o.p.\ there is no $(K,K')$--bounded
$2$--complex~$Y$ with $L$ fixed paths fulfilled by $R$ in such a way that the combinatorial
map to $X/G$
restricts on $\mathrm{Fix}(Y)$ to one of the labellings
assigned to $Y$ by the scheme, and satisfying
$$\mathrm{Cancel}(Y)+|\mathrm{Fix}(Y)|>(d+\varepsilon)|Y|l.$$
\end{prop}

We will use only the following consequence of
Proposition~\ref{prop:isoperimetry2}.

\begin{cor}
\label{cor:rel_notintile} Let $d<\frac{1}{4}$. Consider $2$--complexes $Y'\subset X$ with
$2\leq|Y'|\leq K$, fulfilled by $R$ in such a way
that exactly one $2$--cell $C'\subset Y'$ is carried by the map to $X/G$ onto the $2$--cell corresponding to a specified relator $r_1$.
W.o.p.\ there is no such $Y'$ satisfying
$$\mathrm{Cancel}(Y')>\frac{1}{4}(|Y'|-1)l.$$
\end{cor}
\begin{proof}
Let $K'=\frac{1}{2}K(K-1)$ and $L=K-1$. We apply Proposition~\ref{prop:isoperimetry2} to the random
presentation with relators $R-\{r_1\}$, which are independent
from $r_1$. Consider the polynomial labelling scheme assigning
the labellings that restrict on
each of the $L$ paths of $\mathrm{Fix}(Y)$ to subwords of the cyclic translates of $r_1$.

Given $Y'$ as in the statement of
Corollary~\ref{cor:rel_notintile},
we consider $Y\subset Y'$ that is the closure of the $2$--cells distinct
from $C'$. Let $\mathrm{Fix}(Y)=C'\cap Y$.
By Lemma~\ref{lem:cells_intersection_connected}, the
intersection $C'\cap C$ is connected for any $2$--cell $C\subset
Y$, and thus $\mathrm{Fix}(Y)$ is a union of at most $L=K-1$ subpaths of the
boundary paths of $2$--cells.
We have $\mathrm{Cancel}(Y')=\mathrm{Cancel}(Y)+|\mathrm{Fix}(Y)|$.
Moreover, since $Y'$ is fulfilled by $R$ in such a way that
only $C'$ is carried by the map to $X/G$ onto $r_1$, the
$2$--complex $Y$ is fulfilled by $R-\{r_1\}$ in such a way that
that the restriction to $\mathrm{Fix}(Y)$ is one of the labellings
assigned to $Y$ by our scheme. Thus the desired inequality
follows from the one in Proposition~\ref{prop:isoperimetry2}.
\end{proof}

\section{Tiles}
\label{sec:tiles}
In this section we describe the construction of \emph{tiles} mentioned in the Introduction.
From now on we always assume $d<\frac{1}{4}$.

\begin{defin}
\label{def:tile}
A \emph{tile} $T$ is a single $2$--cell or a $2$--complex $T$ that is the closure of its $2$--cells,
satisfies $$\mathrm{Cancel}(T)>\frac{1}{4} (|T|-1)l,$$
and can be expressed as a union of two tiles which do not share a $2$--cell.
A tile \emph{in $X$} is a tile that is a subcomplex of $X$.
\end{defin}

\begin{rem}
\label{rem:gluingtiles} Let $T,T'$ be tiles in $X$ that do not
share $2$--cells. If $|T\cap T'|> \frac{1}{4}l$, then by
Remark~\ref{rem:cancel} the union $T\cup T'$ is a tile. In the
case where $T,T'$ are single $2$--cells, conversely, if $T\cup
T'$ is a tile, then $|T\cap T'|> \frac{1}{4}l$.
\end{rem}

\begin{rem}
\label{rem-tile-size-bound} If $T$ is a tile in $X$, by
Proposition~\ref{prop:isoperimetry} for each $K,\varepsilon>0$
w.o.p.\ if $|T|\leq K$, then we have
$(d+\varepsilon)|T|>\frac{1}{4}(|T|-1)$. It follows, since
$d<\frac{1}{4}$, that w.o.p.\ the size $|T|$ of a tile is
uniformly bounded. Explicitly, if $d <
\frac{1}{4}\frac{N}{N+1}$, then $|T|\leq N$, since it suffices
to consider $K=2N$ to exclude the possibility of obtaining
tiles by gluing two tiles of size $\leq N$. In particular, for
$d<\frac{5}{24}$ we have $|T|\leq 5$.
\end{rem}

However, the reader will see that the tiles effectively
considered in the article will have size $\leq 4$.

We now generalise Lemma~\ref{lem:cells_intersection_connected}.

\begin{lem}
\label{lem:tile_intersection_connected}
Let $T,T'$ be intersecting tiles in $X$ that do not share $2$--cells. Then $T\cap T'$ is connected.
\end{lem}

Before we give the proof, we deduce the following:

\begin{rem}
\label{rem:short_tree}
By Proposition~\ref{prop:isoperimetry} applied to $T\cup T'$ we have
\begin{align*}
	|T \cap T'| & = \mathrm{Cancel}(T \cup T') - \mathrm{Cancel}(T) - \mathrm{Cancel}(T') \\
		& < \frac14(|T|+|T'|)l - \frac14 (|T|-1)l - \frac14 (|T'|-1)l = \frac12 l.
\end{align*}
By Corollary~\ref{cor:short_geodesic}, $T\cap T'$ is a forest,
hence a tree by Lemma~\ref{lem:tile_intersection_connected}. It
follows that tiles in $X$ are contractible.
\end{rem}

\begin{proof}[Proof of Lemma~\ref{lem:tile_intersection_connected}]
If $T\cap T'$ is not connected, then there is in $T\cup T'$ a homotopically
non-trivial embedded closed path $\alpha\cup\alpha'$ with
$\alpha$ in $T$ and $\alpha'$ in $T'$. Let $D\rightarrow X$ be
a disc diagram for $\alpha\cup\alpha'$. By
Remark~\ref{rem-tile-size-bound}, the size of $T\cup T'$ is
uniformly bounded, hence $|\alpha\cup\alpha'|$ is uniformly
bounded as well. By Theorem~\ref{thm:isoperimetry}, $|D|$ is
uniformly bounded. After passing to a subdisc of $D$, and
allowing $\alpha,\alpha'$ to be immersed, we can also assume
that the cells in $D$ adjacent to $\alpha$, respectively
$\alpha'$, are not mapped to $T$, respectively $T'$.

Let $Y$ be the union of $T\cup T'$ with the image of $D$ in
$X$. The size of $Y$ is uniformly bounded, so we will be able
to apply Proposition~\ref{prop:isoperimetry} to $Y$. Let
$\mathcal{C}$ be the set of $2$--cells of $Y-T\cup T'$. Let $P$
be the image of $\partial D$ in $Y$. We estimate
$\mathrm{Cancel}(Y)$ using Remark~\ref{rem:cancel} with
$\{Y_i\}=\{T,T'\}\cup \mathcal C$. The edges of $P$ contribute
$\frac{1}{2}|P|$ in total to the terms with $Y_i=T,T'$.
Boundary paths of the $2$--cells of $\mathcal{C}$ contribute
additionally $\frac{1}{2}|\mathcal{C}|l$ in total to their own
terms. By Corollary~\ref{cor:no_short_image} we have $|P|\geq
l$. Thus
\begin{align*}
\mathrm{Cancel}(Y) &\geq \mathrm{Cancel}(T)+\mathrm{Cancel}(T')+ \frac{1}{2}|P|+\frac{1}{2}|\mathcal{C}|l \\
&>\frac{1}{4}(|T|-1)l+\frac{1}{4}(|T'|-1)l+\frac{1}{2}l+\frac{1}{4}|\mathcal{C}|l= \frac{1}{4}|Y|l,
\end{align*}
which contradicts Proposition~\ref{prop:isoperimetry}.
\end{proof}

\begin{defin}
\label{def:tileassignment} A \emph{tile assignment} $\mathcal
T$ assigns $G$--equivariantly to each $2$--cell $C$ of $X$ a
tile $\mathcal{T}(C)$ in $X$ containing $C$. An example of a
tile assignment is $\mathcal{T}_0(C)=C$ consisting of single
$2$--cells. If $T=\mathcal{T}(C)$ for some $C$ of $X$, we say
that $T$ \emph{belongs} to $\mathcal{T}$ and write $T\in
\mathcal{T}$.
\end{defin}

\begin{constr}
\label{constr:main} We will make use of a particular tile
assignment $\mathcal{T}=\mathcal{T}_k$ obtained as a last tile
assignment in a sequence $\mathcal{T}_0, \mathcal{T}_1,\ldots,
\mathcal{T}_k$, where $\mathcal{T}_0$ is as in
Definition~\ref{def:tileassignment} and $\mathcal{T}_{i+1}$ is
constructed from $\mathcal{T}_i$ in the following process
consisting of Step~1 and Step~2. During Step~1 of the process
every $2$--cell of $X$ will be in exactly one $T\in
\mathcal{T}_{i+1}$.

\begin{step}
For $i=0,1,\ldots$ we repeat the following construction of $\mathcal{T}_{i+1}$, while there are distinct $T,T'\in\mathcal{T}_i$
satisfying $|T|+|T'|\leq 4$ and $|T\cap T'|> \frac{1}{4}l$.

Choose $T,T'$ so that $|T|+|T'|$ is maximal possible, this
means in particular that if $T$ is a single $2$--cell, then we
first consider $T'$ consisting of two $2$--cells, rather than
$T'$ that is a single cell. This will be used only later in
Proposition~\ref{prop:wallconstruction}. By
Remark~\ref{rem:gluingtiles}, the union $T\cup T'$ is a tile.

We claim that the tiles $T,T'$ are not in the same $G$--orbit.
Otherwise if $T'=gT$, then let $Y$ be the $2$--complex obtained from $T\cup T'$ by identifying $T$ with $T'$.
In other words, $Y$ is obtained from $T$ by identifying for all the pairs of $2$--cells
$C,C'$ of $T$ the paths $C\cap g(C')$ and $g^{-1}(C)\cap C'$. Thus $Y$ is
$\big(|T|,\frac{1}{2}(|T|(|T|-1))+|T|^2\big)$--bounded. Since $\mathrm{Cancel}(Y)=\mathrm{Cancel}(T)+|T\cap T'| >
\frac14 |T| l$, this contradicts Proposition~\ref{prop:isoperimetry}, justifying the claim.

Let $\mathcal{T}_{i+1}$ be obtained from
$\mathcal{T}_{i}$ by differing it only on $gC$ for all $g\in G$
and $\mathcal{T}_i(C)\in \{T,T'\}$ and putting
$\mathcal{T}_{i+1}(gC)=gT\cup gT'$. Loosely speaking, we glue
the tiles $T$ and $T'$.
\end{step}

The process in Step~1 terminates, since the tiles have bounded
size and hence there are finitely many tile orbits. Once this
process terminates, we initiate the process described in
Step~2:

\begin{step}
Repeat the following construction of $\mathcal{T}_{i+1}$, while
there are $T\in\mathcal{T}_i$ with $|T|=2$ and a $2$--cell
$C=\mathcal{T}_i(C)$ such that $T'=T\cup C$ is a tile.

Note that by Step~1 we have $|C\cap T|\leq \frac{1}{4}l$. Let
$C$ be chosen so that $|C\cap T|$ is maximal possible. Consider
first the case where there is a $2$--cell
$C'=\mathcal{T}_i(C')\neq C$ such that $|C'\cap
T'|>\frac{1}{4}l$. The $2$--cells $C,C'$ cannot be in the same
$G$--orbit, otherwise the complex obtained from $T'\cup C'$ by identifying $C$ with $C'$ would
violate Proposition~\ref{prop:isoperimetry}. In that case let
$\mathcal{T}_{i+1}$ be obtained from $\mathcal{T}_{i}$ by
redefining
$\mathcal{T}_{i+1}(gC)=\mathcal{T}_{i+1}(gC')=gT'\cup gC'$. In
the case where there is no such $C'$, we redefine only
$\mathcal{T}_{i+1}(gC)=gT'$. Note that we keep
$\mathcal{T}_{i+1}(gC'')=gT$ for a $2$--cell $C''$ of $T$.
\end{step}
\end{constr}

\begin{rem}
\label{rem:core} Each tile $T\in \mathcal{T}$ obtained in
Construction~\ref{constr:main} contains a unique tile
$T_c\in\mathcal{T}$ that also belongs to the tile assignment in
which we terminate after Step~1. If $T_c\subsetneq T$, then
$|T_c|=2$. We call $T_c$ the \emph{core} of $T$. If distinct
$T,T'\in \mathcal{T}$ share $2$--cells, then these are the two
$2$--cells of $T_c=T'_c$ with $|T_c|=2$ (because in Step~2 we worked only with $|T|=2$).
\end{rem}

\section{Tile-walls}
\label{sec:tile-walls}

In this section we will extend the hypergraph construction from
the strategy outline in the Introduction to all the tiles in
the tile assignment $\mathcal{T}$ from
Construction~\ref{constr:main}. Recall our standing assumption
$d<\frac{1}{4}$.

\begin{defin}
\label{def:tile-wall} Let $T$ be a tile. A \emph{tile-wall
structure} on $T$ is an equivalence relation $\sim_T$ on the
edge midpoints of $T$, such that:
\begin{itemize}
\item The relation $\sim_T$ restricts to the boundary path
    of each $2$--cell $C$ of $T$ to a relation $\sim_C$
    that has exactly $2$ elements in each equivalence
    class.
\item For each equivalence class $\mathcal{W}$ of $\sim_T$,
    called a \emph{tile-wall}, consider the graph
    $\Gamma_{\mathcal{W}}$ in $T$, obtained by connecting
    the points of $\mathcal{W}$ in the boundary path of
    each $2$--cell $C$ by a diagonal in $C$. We call
    $\Gamma_{\mathcal{W}}$ the \emph{hypergraph} of
    $\mathcal{W}$, and require that it is a tree.
\end{itemize}
If $x\sim_Tx'\in \mathcal{W}$, then the unique path from $x$ to
$x'$ in $\Gamma_{\mathcal{W}}$ is called the \emph{hypergraph
segment} between $x$ and $x'$ and is denoted by $xx'$.
\end{defin}

\begin{defin}
\label{def:balance}
Let $T$ be a $2$--complex. The \emph{balance} of $T$ is the value
$$\mathrm{Bal}(T)=\frac{1}{4}(|T|+1)l-\mathrm{Cancel}(T).$$
Note that if $T$ is a tile, then $\mathrm{Bal}(T)\leq \frac{1}{2}l$ by Definition~\ref{def:tile}.
Moreover, if $T$ is a tile in $X$, then
$\mathrm{Bal}(T)>\frac{1}{4}l$ by Proposition~\ref{prop:isoperimetry}, since $d<\frac{1}{4}$.
\end{defin}

\begin{defin}
\label{def:balanced} Let $C$ be a $2$--cell in a tile $T$. A
tile-wall structure on $T$ is \emph{$C$--balanced} if for each
tile-wall $\mathcal{W}$ and $x,x'\in\mathcal{W}$ such that the
hypergraph segment $xx'$ traverses $C$, the distance between $x$ and $x'$
in $T^{(1)}$ satisfies $$|x,x'|_T\geq \mathrm{Bal}(T).$$ For
example, a tile-wall structure on a single $2$--cell $C$ is
$C$--balanced if and only if $\sim_C$ is the antipodal
relation. We say that a tile-wall structure on $T$ is
\emph{balanced} if it is $C$--balanced for every $2$--cell $C$
in $T$.
\end{defin}

Before we construct balanced tile-walls in Example~\ref{ex:main} and Proposition~\ref{prop:wallconstruction}, we need a handful of lemmas.

\begin{lem}
\label{lem:walls_do_not_backtrack} Let $T,T'$ be tiles in $X$
that do not share $2$--cells, and suppose that $T$ has a
$C$--balanced tile-wall structure $\sim_T$, for
some $2$--cell $C$ in $T$. Let $x\sim_Tx'$,
such that $xx'$ traverses $C$. Then at most one of $x,x'$ lies
in $T'$.
\end{lem}
In particular, if the tile-wall structure is balanced, then the conclusion holds for all distinct $x\sim_Tx'$.

\begin{proof}
If both $x,x'$ lie in $T'$, then by Lemma~\ref{lem:tile_intersection_connected} we have $|T\cap T'|\geq|x,x'|_T\geq \mathrm{Bal}(T)$. Thus
\begin{align*}
\mathrm{Cancel}(T\cup T')&=|T\cap T'|+\mathrm{Cancel}(T)+\mathrm{Cancel}(T') \\
& \geq \frac{1}{4}(|T|+1)l+\frac{1}{4}(|T'|-1)l=\frac{1}{4}(|T\cup T'|)l,
\end{align*}
which contradicts Proposition~\ref{prop:isoperimetry}.
\end{proof}

\begin{lem}
\label{lem:wall_in_subtile} Let $T,T'$ be tiles in $X$ that do
not share $2$--cells, with $|T\cap T'|\geq \frac{1}{4}l$.
Suppose that $T$ has a $C$--balanced tile-wall structure, for
some $2$--cell $C$ in $T$. Let $x\sim_Tx'$, such that $xx'$
traverses $C$. Then
$$|x,x'|_{T\cup T'}\geq \mathrm{Bal}(T\cup T')+|T\cap T'|-\frac{1}{4}l.$$
\end{lem}
\begin{proof}
By Remark~\ref{rem:short_tree} we have $|T\cap T'|< \frac{1}{2}l$. Hence by Corollary~\ref{cor:short_geodesic} we obtain
$$|x,x'|_{T\cup T'}=|x,x'|_{T}\geq \mathrm{Bal}(T).$$

On the other hand,
\begin{align*}
\mathrm{Bal}(T\cup T')&=\frac{1}{4}(|T\cup T'|+1)l-\mathrm{Cancel}(T\cup T')\\
&=\frac{1}{4}(|T|+1)l +\frac{1}{4}|T'|l-\big(\mathrm{Cancel}(T)+|T\cap T'|+\mathrm{Cancel}(T')\big)\\
&\leq\mathrm{Bal}(T)-|T\cap T'|+\frac{1}{4}l.\qedhere
\end{align*}
\end{proof}

\begin{lem}
\label{lem:mix_gluing} Let $T,T'$ be tiles in $X$ that do not
share $2$--cells, with tile-wall structures that are
$C$--(respectively $C'$--)balanced. Let $\alpha$ be an embedded
path in $T\cap T'$ of length $\leq \frac{1}{4}l$ such that
$T\cap T'$ is contained in the $\frac{1}{4}l$--neighbourhood of
$\alpha$. Let $s\colon \alpha\rightarrow \alpha$ be the
symmetry exchanging the endpoints of $\alpha$. Suppose that we
have edge midpoints $x\in T,x'\in T', y\in \alpha$ such that
$x\sim_Ty, x'\sim_{T'}s(y)$, where $xy, x's(y)$ traverse
$C,C'$, respectively. Then
$$|x,x'|_{T\cup T'}\geq \mathrm{Bal}(T\cup T').$$
\end{lem}

In the proof we need the following:

\begin{sublemma}
\label{sub:mix_gluing} Let $A$ be a tree, $\alpha\subset A$ a
path such that $A$ is contained in the $q$--neighbourhood of
$\alpha$. Let $s$ be the symmetry of $\alpha$ exchanging its
endpoints. Then for any points $z,z'\in A$ and $y \in \alpha$
we have
$$|y,z|_A+|s(y),z'|_A\leq |A|+\max\{|\alpha|,q\}.$$
\end{sublemma}
\begin{proof}
First consider the case where the paths $yz, s(y)z'$ in $A$
intersect outside $\alpha$. Then they leave $\alpha$ in the
same point, and hence $|yz\cap\alpha|+|s(y)z'\cap\alpha|\leq
|\alpha|$. Their length outside $\alpha$ is bounded by both $q$
and $|A|-|\alpha|$. Thus $|y,z|_A+|s(y),z'|_A\leq
|\alpha|+(q+|A|-|\alpha|)$, as desired. In the second case,
where $yz, s(y)z'$ are allowed to intersect only in $\alpha$,
we have $|y,z|_A+|s(y),z'|_A\leq 2|\alpha|+(|A|-|\alpha|)$.
\end{proof}

\begin{proof}[Proof of Lemma~\ref{lem:mix_gluing}]
We apply Sublemma~\ref{sub:mix_gluing} with $A=T\cap T'$. The
upper bound from Sublemma~\ref{sub:mix_gluing} is $\leq
|A|+\frac{1}{4}l$. Let $z,z'$ be the closest point projections
to $A$ of $x,x'$ in the $1$--skeleton of $T\cup T'$. By
Sublemma~\ref{sub:mix_gluing}, we have $|y,z|+|s(y),z'|\leq
|A|+\frac{1}{4}l$. Then
\begin{align*}
|x,x'|_{T\cup T'}&\geq |x,z|_{T}+|x',z'|_{T'}\geq |x,y|_T-|y,z|_T+|x',s(y)|_{T'}-|s(y),z'|_{T'}\\
&\geq \frac{1}{4}(|T|+1)l-\mathrm{Cancel}(T)+\frac{1}{4}(|T'|+1)l-\mathrm{Cancel}(T')-\Big(|A|+\frac{1}{4}l\Big)\\
&=\frac{1}{4}(|T\cup T'|+1)l-\mathrm{Cancel}(T\cup T'),
\end{align*}
as desired (the last equality comes from Remark~\ref{rem:cancel}).
\end{proof}

Applying Lemma~\ref{lem:mix_gluing} with $\alpha$ equal to a point $y$, we obtain the following.
Note that the distance condition on $y$ is satisfied automatically if $|T\cap T'|\leq \frac{1}{4}l$.

\begin{cor}
\label{lem:ant_gluing} Let $T,T'$ be tiles in $X$ that do not
share $2$--cells, with tile-wall structures that are
$C$--(respectively $C'$--)balanced. Let $y\in T\cap T'$ be an
edge midpoint such that $T\cap T'$ is contained in the
$\frac{1}{4}l$--neighbourhood of $y$. Suppose that we have edge
midpoints $x\in T,x'\in T'$ satisfying $x\sim_Ty,
x'\sim_{T'}y$, where $xy, x'y$ traverse $C,C'$, respectively.
Then
$$|x,x'|_{T\cup T'}\geq \mathrm{Bal}(T\cup T').$$
\end{cor}

The following warm-up example generalises the balanced tile-wall construction from the Introduction.

\begin{ex}
\label{ex:main} Let $T$ be a tile and let $T'$ be a complex
obtained by gluing to $T$ a $2$--cell $C$ along a path $A$ of
length $\frac{1}{4}l<|A|<\frac{1}{2}l$. Suppose that $T$ has a
balanced tile-wall structure $\sim_T$. We can then extend
$\sim_T$ to the following balanced tile-wall structure
$\sim_{T'}$.

Let $\alpha_+,\alpha_-\subset A$ be subpaths of length
$\lceil|A|-\frac{1}{4}l\rceil$ starting at the endpoints of $A$. The paths
$\alpha_\pm$ are disjoint since $|A|<\frac{1}{2}l$. Let
$\beta_+,\beta_-$ be the images in $\partial C$ of
$\alpha_+,\alpha_-$ under the antipodal map. Note that
$\beta_\pm$ are outside $T$ since $|A|<\frac{1}{2}l$. Let $s_+$
be the symmetry of $\alpha_+$ exchanging its endpoints, and let
$s_-$ be the symmetry of $\alpha_-$ exchanging its endpoints.

We define $\sim_C$ to be the antipodal relation outside the
union of the interiors of $\alpha_+,\alpha_-,\beta_+$ and
$\beta_-$. For an edge midpoint $x$ in the interior of
$\beta_\pm$ and its antipode $y\in \alpha_\pm$ we define
$x\sim_Cs_\pm(y)$. By Lemma~\ref{lem:walls_do_not_backtrack}
for each pair of edge midpoints related by $\sim_{T}$ at most
one of them lies in $A$, and by construction the same holds for
$\sim_C$. Thus the relation $\sim_{T'}$ generated by $\sim_{T}$
and $\sim_C$ is a tile-wall structure.

Now we show that the relation $\sim_{T'}$ is balanced. Consider
distinct $x\sim_{T'}x'$. If $x,x'\in T$, then by
Lemma~\ref{lem:wall_in_subtile} we have
$|x,x'|_{T'}\geq\mathrm{Bal}(T')$, as desired. Secondly,
consider the case where $x,x'\in C$. If $x,x'$ are not
antipodal, then one of them, say $x$, lies in $\alpha_\pm$, so
the antipode of $x'$ is $s_\pm(x)$. By
Lemma~\ref{lem:wall_in_subtile} we have thus
$|s_\pm(x),x'|_{T'}\geq\mathrm{Bal}(T')
+|A|-\frac{1}{4}l>\mathrm{Bal}(T')+|x,s_\pm(x)|_{T'}$, and by
the triangle equality we obtain the desired bound on
$|x,x'|_{T'}$.

Finally, consider the case where $x\in T-C, x'\in C-T$. Thus
there is $y\in A$ with $x\sim_{T}y$ and $y\sim_{C}x'$. If $y\in
\alpha_\pm$, then the required estimate follows from
Lemma~\ref{lem:mix_gluing}. Otherwise $y$ and $x'$ are
antipodal and we use Corollary~\ref{lem:ant_gluing}.
\end{ex}

Now follows the key result of the article, where we construct
$C$--balanced tile-wall structures on all the tiles from the
tile assignment in Construction~\ref{constr:main}. In Part~1 we
consider tiles obtained in Step~1 of that construction,
extending Example~\ref{ex:main}. In Part~2, we need to deal
with tiles obtained in Step~2, which might share $2$--cells
according to Remark~\ref{rem:core}. To deal with this
complication in later sections we need to record additional
ad-hoc properties (ii)---(iii) in
Proposition~\ref{prop:wallconstruction}, which we recommend to
ignore at a first reading.

\begin{prop}
\label{prop:wallconstruction} For the collection of tiles $T$
belonging to the tile assignment $\mathcal{T}$ from
Construction~\ref{constr:main}, there are tile-wall structures
$\sim_T$ that are $C$--balanced for each $C$ with
$T=\mathcal{T}(C)$. Moreover:
\begin{enumerate}[(i)]
\item
The relation $\sim_{T}$ on $T\in \mathcal{T}$ restricts to $\sim_{T_c}$ on the core $T_c\in \mathcal{T}$ from Remark~\ref{rem:core}.
\item If $C'$ is a $2$--cell of $T-T_c$ with $x\sim_{C'}y$
    distinct and not antipodal in $C'$, then one of $x,y$,
    say $y$, lies in $T_c$ and the edge midpoint $y'$
    antipodal to $x$ also lies in $T_c$.
\item If $C'$ is a $2$--cell of $T-T_c$ with
    $x\sim_Ty\sim_Tw$, where $x\neq y\in C'$ and $w\in
    T_c$, then one of $x,y$ lies in $T_c$ and the other
    lies in no other $2$--cells of $T$ except for $C'$.
\end{enumerate}
\end{prop}

\begin{proof}
Recall that in Construction~\ref{constr:main} we obtain
$\mathcal{T}=\mathcal{T}_k$ as the last of a sequence of tile
assignments $(\mathcal{T}_i)$. We will construct inductively
relations $\sim^i_T$ on the tiles $T\in\mathcal{T}_i$
satisfying required conditions for $\mathcal{T}=\mathcal{T}_i$.
More precisely, for all $2$--cells of $X$ we will construct
$\sim^i_C$ generating $\sim^i_T$, in particular assertion (i)
will be automatic. Note that for $\mathcal{T}=\mathcal{T}_0$,
where $\mathcal{T}_0(C)=C$ it suffices to consider the
antipodal relation.

\smallskip

\noindent \textbf{Part~1.} During Step~1 of
Construction~\ref{constr:main}, distinct tiles in
$\mathcal{T}_{i+1}$ do not share $2$--cells, and for each
$2$--cell $C$ of $T\in \mathcal{T}_{i+1}$ we have
$\mathcal{T}_{i+1}(C)=T$. Thus if $T,T'\in \mathcal{T}_i$ are as in
Step~1 of Construction~\ref{constr:main}, we only need to
construct a tile-wall structure on $T\cup T'$ that is balanced
(assertions (ii)--(iii) are void). If $|T|+|T'|\leq 3$, then at
least one of $T,T'$ is a single cell and such a tile-wall
structure is given in Example~\ref{ex:main}.

Now assume that in Step~1 we have $|T|=|T'|=2$. Without loss of
generality assume that $T$ appeared for smaller $i$ in
$\mathcal{T}_i$ than $T'$. Denote the $2$--cells of $T'$ by
$C_1,C_2$. Note that the intersection path $\alpha_j=C_j\cap T$
cannot have length $>\frac{1}{4}l$: otherwise, by the
maximality condition in Step~1, instead of gluing $C_1$ to
$C_2$ to obtain $T'$ we would have had to glue $C_j$ to $T$. In
particular, the intersection $T\cap T'$ has the form of a
(possibly degenerate) tripod $\alpha_1\cup\alpha_2$, where an
endpoint of $\alpha_1$ coincides with an endpoint of
$\alpha_2$, and the other endpoint $u_1$ of $\alpha_1$
(respectively $u_2$ of $\alpha_2$) is outside $\alpha_2$
(respectively $\alpha_1$). Moreover, the complement in
$\alpha_1\cup\alpha_2$ of the $\frac{1}{4}l$--neighbourhood of
$u_2$ (respectively $u_1$) is either empty or is a path containing $u_1$ ($u_2$)
disjoint from $\alpha_2$ ($\alpha_1$). This path is an edge-path if $l$ is divisible by $4$, otherwise it ends with a
half-edge. Its span, which is an edge-path, will be called
$\alpha_+$ ($\alpha_-$).

We change the relation $\sim^i_{C_1}$ (which does not have to
be antipodal at this stage) to $\sim^{i+1}_{C_1}$ in the
following way. Let $s_+$ be the symmetry of $\alpha_+$ exchanging its
endpoints. If we
have distinct $x\sim^i_{C_1} y$ with $y$ in the interior of
$\alpha_+$, then we replace it with $x\sim^{i+1}_{C_1}s_+(y)$.
Analogically, let $s_-$ be the symmetry of $\alpha_-$ exchanging its
endpoints. If we
have distinct $x\sim^i_{C_2} y$ with $y$ in the interior of
$\alpha_-$, then we replace it with $x\sim^{i+1}_{C_2}s_-(y)$.
All other relations remain unchanged.

By Lemma~\ref{lem:walls_do_not_backtrack}, the relation
$\sim^{i+1}_{T\cup T'}$ generated by
$\sim^i_T,\sim^{i+1}_{C_1}$ and $\sim^{i+1}_{C_2}$ is a
tile-wall structure. It is balanced by
Lemmas~\ref{lem:wall_in_subtile},~\ref{lem:mix_gluing}, and
Corollary~\ref{lem:ant_gluing}, by considering the same four
cases as in Example~\ref{ex:main}. This closes the construction
of tile-walls for the tiles in $\mathcal{T}_{i+1}$ from Step~1
of Construction~\ref{constr:main}.

\smallskip

\noindent \textbf{Part~2.} Now consider $T,C,C' \in
\mathcal{T}_{i}$ as in Step~2 of Construction~\ref{constr:main}
($C'$ might not be defined). Note that by the process in Step~1,
we have, when defined, all $|T\cap C|,|T\cap C'|, |C\cap
C'|\leq \frac{1}{4}l$. Consequently $|T\cap (C\cup C')|,
|C\cap(C'\cup T)|\leq \frac{1}{2}l$.

We first claim that the tile-wall structure $\sim^{i+1}_{T'}$
on $T'=T\cup C$
generated by~$\sim^{i}_{T}$ and the antipodal relation
$\sim^{i}_{C}$ is $C$--balanced. Indeed, suppose that
$x\sim^{i+1}_{T'}x'$ and that $xx'$ traverses $C$. Then without
loss of generality we have $x'\in C-T$. If $x\in T-C$, then
$|x,x'|_{T'}\geq \mathrm{Bal}(T')$ by
Corollary~\ref{lem:ant_gluing}. Otherwise $x$ is the antipode
of $x'$ in $C$, so we have trivially $|x,x'|_{T'}=\frac{1}{2}l$
which is $\geq \mathrm{Bal}(T')$ by
Definition~\ref{def:balance}. This justifies the claim.

If we continue to glue a $2$--cell $C'$ to $T'$, let
$A$ be the path $C'\cap T'$ of length $|A|>\frac{1}{4}l$. Note
that by Lemma~\ref{lem:tile_intersection_connected} the path
$A$ consists of three segments, the first one in $T-C$, the
second one (possibly degenerate) in $T\cap C$, and the third
one in $C-T$. Let $\alpha\subset A$ be the subpath of length
$\lceil|A|-\frac{1}{4}l\rceil$ containing that endpoint of $A$ which lies
in $T$. Since $A\cap C=C'\cap C$ has length $\leq
\frac{1}{4}l$, the interior of the path $\alpha$ is disjoint from~$C$. We set
$\sim^{i+1}_{C'}$ to be antipodal except in the interior of
$\alpha$ and its antipodal image $\beta$, where for antipodal
$x\in \beta, y\in \alpha$ we put $x\sim^{i+1}_{C'}s(y)$, where
$s$ is the symmetry of $\alpha$ exchanging its
endpoints.

Let $T''=T'\cup C'$. By Lemma~\ref{lem:walls_do_not_backtrack},
the relation $\sim^{i+1}_{T''}$ generated by $\sim^{i+1}_{T'}$
and $\sim^{i+1}_{C'}$ is a tile-wall structure. We now prove
that $\sim^{i+1}_{T''}$ is $C$--balanced and $C'$--balanced.
Let $x\sim^{i+1}_{T''} x'$ with the hypergraph segment $xx'$
traversing $C$ or $C'$. If $xx'$ is contained in $T'$ or $C'$,
then the required estimate follows from the claim above and
from Lemma~\ref{lem:wall_in_subtile}, as in the first two cases
of Example~\ref{ex:main}.

Otherwise we can assume $x\in T'-C', x'\in C'-T'$, and there is
$y\in A\cap xx'$. Note that if the neighbourhood of $y$ in $yx$
lies in $C$, then we have $yx\subset C$ since the length of
$|C\cap (T\cup C')|$ is $\leq \frac{1}{2}l$ and
$\sim^{i+1}_C=\sim^{i}_C$ was antipodal. In this case
$|x,x'|_{T''}=|x,x'|_{C\cup C'}$ by
Corollary~\ref{cor:short_geodesic}, and the latter is $\geq \frac{1}{2}l$
by Corollary~\ref{lem:ant_gluing} applied with $C$ and $C'$
playing the roles of $T,T'$.

Otherwise, the neighbourhood of $y$ in $yx$ lies in $T$. If,
nevertheless, $yx$ traverses $C$, then since $\sim^{i+1}_{T'}$
is $C$--balanced, we can apply Lemma~\ref{lem:mix_gluing} and
Corollary~\ref{lem:ant_gluing} as in the last two cases of
Example~\ref{ex:main} to obtain $|x,x'|_{T''}\geq
\mathrm{Bal}(T'')$.

It remains to consider the situation where $yx\subset T$. By
Lemma~\ref{lem:mix_gluing} and Corollary~\ref{lem:ant_gluing}
we obtain $|x,x'|_{T\cup C'}\geq \mathrm{Bal}(T\cup C')$. By
Corollary~\ref{cor:short_geodesic}, we have
$|x,x'|_{T''}=|x,x'|_{T\cup C'}$. By the process in Step~2 we
have $|C\cap T|\geq |C'\cap T|$, so that $|C\cap (C'\cup
T)|\geq |C'\cap T'|>\frac{1}{4}l$. Thus
$$\mathrm{Bal}(T'')=\mathrm{Bal}(T\cup C')+\frac{1}{4}l-|C\cap (C'\cup T)|<\mathrm{Bal}(T\cup C')\leq |x,x'|_{T''},$$
as desired.

Assertions (ii) and (iii) follow immediately from the construction.
\end{proof}

In the next section we will operate on tiles which we will need to make disjoint.
To do this, we will sometimes replace them by single $2$--cells, according to the behaviour of the wall in which we will be
interested:

\begin{defin}
\label{def:augmented}
Let $\mathcal{T}$ be the tile assignment from Construction~\ref{constr:main}. Let $C$ be a $2$--cell of $X$ and $\mathcal{W}$ a wall of $T=\mathcal{T}(C)$ intersecting $C$.
We assign to each such pair $(C,\mathcal{W})$, the \emph{augmented} tile denoted by $\mathcal{T}(C,\mathcal{W})$ that equals
\begin{itemize}
\item
$T$ if $\mathcal{W}$ intersects the core $T_c$ of $T$, and
\item
$C$ otherwise.
\end{itemize}
If $\gamma\subset C$ is a hypergraph segment of a wall $\mathcal{W}$ of $T$, then we denote the augmented tile $\mathcal{T}(C,\mathcal{W})$ also by $\mathcal{T}(C,\gamma)$.
\end{defin}

\begin{rem}
\label{rem:smallantipodal} Suppose that we have a $C$--balanced
tile-wall structure on $T$ satisfying
Proposition~\ref{prop:wallconstruction}(ii). If
$\mathcal{T}(C,\mathcal{W})=C$, then the two points of
$\mathcal{W}$ in $C$ are antipodal. Thus in general any
$x,y\in\mathcal{W}$ in $T'=\mathcal{T}(C,\mathcal{W})$ such
that $xy$ traverses~$C$ satisfy $|x,y|_{T'}\geq
\mathrm{Bal}(T')$.
\end{rem}

\section{Walls}
\label{sec:walls}

\begin{defin}
\label{def:wall} Suppose that on each $2$--cell $C$ of the
Cayley complex $X$ we have a relation $\sim_C$ on edge
midpoints that has exactly two elements in each equivalence
class. A \emph{wall structure} on $X$ is the equivalence
relation $\sim$ on edge midpoints of $X$ generated by such
$\sim_C$.

For an equivalence class $\mathfrak{W}$ of $\sim$, called a
\emph{wall}, consider the \emph{hypergraph}
$\Gamma_{\mathfrak{W}}$, immersed in~$X$, obtained by
connecting the points $x\sim_Cx'\in \mathfrak{W}$ in each
$2$--cell $C$ by a diagonal in $C$. A \emph{hypergraph segment}
is an edge-path in $\Gamma_{\mathfrak{W}}$.
\end{defin}

We consider tile-wall structures $\sim_T$ on the tiles in the
tile assignment $\mathcal{T}$ from
Construction~\ref{constr:main} satisfying
Proposition~\ref{prop:wallconstruction}. By
Proposition~\ref{prop:wallconstruction}(i) they restrict to
consistent $\sim_C$ on $2$--cells, and thus give rise to a wall
structure $\sim$ on $X$, which we fix from now on.

\begin{thm}
\label{thm:walls_trees}
If $d<\frac{5}{24}$, then w.o.p.\ all hypergraphs are embedded trees.
\end{thm}

The proof of Theorem~\ref{thm:walls_trees} is divided into two
parts. In the current section we prove its weaker version, that
hypergraphs of a priori bounded length are embedded trees. We
will complete the proof in Section~\ref{sec:quasi-convexity}.

There is a technical preliminary step to perform. To have
control over a (self-intersecting) hypergraph segment and the
tiles it traverses, we will decompose it into particular
subsegments in Definition~\ref{def:decomposition}. Next, we
will improve the decomposition so it becomes tight, see
Definition~\ref{def:tight} and Proposition~\ref{prop:shrink}.
We recommend the reader to skip the proof of the latter at a
first reading.

\begin{defin}
\label{def:decomposition}
A concatenation $\gamma_1\cdots \gamma_n$ forming a hypergraph segment is its \emph{decomposition of length $n$}
if for each $i=1,\ldots, n$, for one of the $2$--cells $C_i$ traversed by $\gamma_i$, $\gamma_i$ is contained in
$T_i=\mathcal{T}(C_i,\gamma_i\cap C_i)$, which is the augmented tile
for $C_i$.
\end{defin}

\begin{rem}
\label{rem:T_i_unique}
Given a decomposition $\gamma_1\cdots \gamma_n$, the tile $T_i$ (but not $C_i$) is uniquely determined
by $\gamma_i$. Otherwise, if we had $\gamma_i\subset T'_i\neq T_i$ with
$T'_i=\mathcal{T}(C'_i,\gamma_i\cap C'_i)$, by Definition~\ref{def:augmented} and Remark~\ref{rem:core} we would have
$T_i,T_i'\in \mathcal{T}$ and $T_i\cap T'_i=T_c$, which is
their common core. Consequently, we would have $\gamma_i\subset T_c=C_i\cup C_i'$, which would yield $T_i=\mathcal{T}(C_i,\gamma_i\cap
C_i)=T_c=\mathcal{T}(C'_i,\gamma_i\cap C'_i)=T_i'$, a contradiction.
\end{rem}

\begin{defin}
\label{def:returning}
Consider a decomposition $\gamma=\gamma_1\cdots \gamma_n$. Denote the endpoints of $\gamma$ by $x_0$
and $x_n$. We say that $\gamma_1\cdots \gamma_n$ is \emph{returning at $T_0$} for a tile
$T_0\in\mathcal{T}$ if $x_0,x_n\in T_0$ and there is no $T_i$
that contains $\bigcup_j T_j$.
\end{defin}

The main goal of this section is to prove the following.

\begin{prop}
\label{prop:main} Let $d<\frac{5}{24}$. For each $N$ w.o.p.\
there is no decomposition $\gamma_1\cdots \gamma_n$ returning
at tile in $\mathcal{T}$ with $n\leq N$.
\end{prop}

Proposition~\ref{prop:main} implies the aforementioned weak
version of Theorem~\ref{thm:walls_trees} in view of the
following observation.

\begin{rem}
\label{rem:diagram} If a hypergraph segment $\gamma$
self-intersects in a $2$--cell $C_0$ of $X$, then let
$C_0,C_1,\ldots, C_n,$ $C_{n+1}=C_0$ be consecutive $2$--cells
traversed by the diagonals $\gamma_i\subset \gamma$. Let
$T_0=\mathcal{T}(C_0)$ and for $1\leq i\leq n$ let
$T_i=\mathcal{T}(C_i,\gamma_i)$. No $T_i$ contains all the
others, since hypergraphs in tiles are embedded trees. Thus
$\gamma_1\cdots \gamma_n$ is returning at $T_0$.
\end{rem}

The tiles $T_1, \ldots, T_n$ of a decomposition $\gamma_1\cdots
\gamma_n$ can share $2$--cells. Before we begin the proof of
Proposition~\ref{prop:main} we will modify the decomposition so
that $T_i$ overlap in the following controlled way.

\begin{defin}
\label{def:tight} A decomposition $\gamma_1\cdots \gamma_n$ is
\emph{tight} if for $1\leq i < j\leq n$ the tiles $T_i, T_j$
share no $2$--cells, except for possibly some pairs $T_i,
T_{i+1}\in \mathcal{T}$ with common core $T_c$, in which case
$\gamma_{i+1}$ is exactly a diagonal of $C_{i+1}$ and
intersects $T_c$ only at its starting point.

If the decomposition is returning at $T_0$, then it is
\emph{tight} if the same holds for $0\leq i<j\leq n$.
\end{defin}

Note that in a tight (returning) decomposition if $T_i,T_{i+1}$
share $2$--cells, then $T_{i+1},T_{i+2}$ do not share
$2$--cells. Indeed, otherwise the core $T_c$ of $T_{i+1}$ would
also be the core of $T_i$ and $T_{i+2}$, hence $T_i$ and
$T_{i+2}$ would also share $2$--cells.

\begin{prop}\label{prop:shrink}
\begin{enumerate}[(i)]
\item For any hypergraph segment with a returning
    decomposition of length~$\leq N$, there is a tight
    returning decomposition of length~$\leq N$ of a
    subsegment of it (possibly with reversed orientation).
\item Any hypergraph segment with decomposition of length
    $\leq N$ has a tight decomposition of length $\leq N$,
    or there is a returning decomposition of length~$\leq
    N$ of a subsegment of it.
\end{enumerate}
\end{prop}

\begin{proof}
We prove (i) by contradiction. Let $\gamma_1\cdots \gamma_n$ be
a returning decomposition of a subsegment of the given segment
with minimal length $n\leq N$. Let $T_0$ be the tile at which
it is returning. Denote the endpoints of $\gamma_i$ by
$x_{i-1}$ and $x_i$. After shortening $\gamma_1$ or $\gamma_n$
we can assume that the neighbourhood of $x_0$ in $\gamma_1$
intersects $T_0$ only at $x_0$ and that the neighbourhood of
$x_n$ in $\gamma_n$ intersects $T_0$ only at $x_n$.

We analyse in what situation $T_i$ and $T_j$ might
share $2$--cells, where $i<j$. Suppose first $T_j\subset T_i$.
If $j=i+1>1$, then we could have merged $\gamma_i\cup
\gamma_{i+1}$ into one segment in $T_i$ to decrease $n$. If
$j=1$ and $i=0$, then this would contradict the assumption on $x_0$.
If $j>i+1$, then we could have replaced
$\gamma_1\cdots \gamma_n$ with $\gamma_{i+1}\cdots\gamma_{j-1}$
and $T_0$ with $T_i$ to decrease $n$ as well. If $T_i\subset
T_j$, the argument is the same, except when $i=0$. Then if $j=1$, we could have passed to
$\gamma_2\cdots \gamma_n$ replacing $T_0$ with $T_1$. If $j>1$,
we could restrict to $\gamma_{1}\cdots\gamma_{j-1}$ replacing
$T_0$ with $T_j$.

This shows that if $T_i,T_j$ share $2$--cells, then neither is
contained in the other. Hence they cannot be single $2$--cells,
and thus are tiles of $\mathcal{T}$ by
Definition~\ref{def:augmented}. By Remark~\ref{rem:core},
$T_i,T_j$ share their core $T_c$. Note that $C_j$ is outside
$T_c$, since otherwise we would have $T_j=T_c\subset T_i$. By
Proposition~\ref{prop:wallconstruction}(iii) there is a unique
edge midpoint $x\in \gamma_j$ in $C_j\cap T_c$.

Consider first $j=i+1$. Without loss of generality we can
assume that $C_{i+1}$ is the first $2$--cell in $T_{i+1}-T_c$
traversed by $\gamma_{i+1}$. Moreover, we can move to
$\gamma_i$ the part of $\gamma_{i+1}$ preceding
$\gamma_{i+1}\cap C_{i+1}$, which lies in $T_c$. Then $x=x_i$,
otherwise we could replace $T_0$ with $T_i$ and pass to
$x_i\cdots x\subset \gamma_{i+1}$, which decreases $n$ unless
$i=0, n=1$ and $x=x_1$ in which case we just interchange $x_0$
with $x_1$. By Proposition~\ref{prop:wallconstruction}(iii) the
edge midpoint $y\in \gamma_{i+1}\cap C_{i+1}$ distinct from $x$
equals $x_{i+1}$. Hence $\gamma_{i+1}$ is a diagonal of
$C_{i+1}$ as in Definition~\ref{def:tight}.

If $j>i+1$, then passing to the hypergraph segment
$\gamma_{i+1}\cdots x$, we obtain a contradiction with
minimality of $n$, unless $i=0, j=n$ and $x=x_n$. Then again by
Proposition~\ref{prop:wallconstruction}(iii), $\gamma_n$ is a
diagonal of $C_n$. Note that, unless $n=1$, it cannot
simultaneously happen that $T_0,T_1$ share $2$--cells and
$T_0,T_n$ share $2$--cells. Otherwise $T_1$ and $T_n$ would
also share the $2$--cells of the common core $T_c$, which would
yield $n=2$ and $x=C_2\cap T_c$ being simultaneously equal to
$x_1$ and $x_2$, a contradiction. In particular, by possibly
reversing at the beginning of the procedure the order of
$\gamma_i$, we can assume that $T_0$ and $T_n$ do not share
$2$--cells, unless $n=1$, which case was discussed above. Thus
the returning decomposition $\gamma_1\cdots \gamma_n$ we
obtained is tight, as desired.

The proof of (ii) is similar. Firstly, by merging some of the
$\gamma_i$ as above we can assume that $T_i$ does not contain
$T_j$ for $i\neq j$. Indeed, if $j=i+1$, then we can merge
$\gamma_i$ with $\gamma_{i+1}$. If $j>i+1$, we obtain a
decomposition $\gamma_{i+1}\cdots \gamma_{j-1}$ returning at
$T_i$, and we are done. Secondly, we can also assume that $T_i$
and $T_j$ share $2$--cells only if $j=i+1$, since otherwise we
also obtain a returning decomposition of a subsegment. Finally,
for $T_{i},T_{i+1}$ sharing $2$--cells, similarly as in the
proof of (i), after moving part of $\gamma_{i+1}$ to
$\gamma_{i}$, $\gamma_{i+1}$ is a diagonal of $C_{i+1}$
satisfying the condition in Definition~\ref{def:tight}.
\end{proof}

Here is the final piece of terminology used in the proof of
Proposition~\ref{prop:main}.

\begin{defin}
\label{def:bounded diagram} Let $\gamma_1\cdots \gamma_n$ be a
decomposition returning at $T_0\in\mathcal{T}$. Denote the
endpoints of $\gamma_i$ by $x_{i-1}$ and $x_i$. A disc diagram
$D\rightarrow X$ \emph{bounded} by $\gamma_1\cdots \gamma_n$
returning at $T_0$ is a disc diagram for
$\alpha_0\alpha_1\cdots \alpha_n$, where $\alpha_i$ is mapped
to $T_i$ and for $i\neq 0$ its endpoints are mapped to
$x_{i-1},x_{i}$. Thus we allow half-edge spurs at $\partial D$.

Likewise, a disc diagram $D\rightarrow X$ \emph{bounded} by a
decomposition $\gamma_1\cdots \gamma_n$ and a path $\alpha$ in
$X^{(1)}$ is a disc diagram for $\alpha\alpha_1\cdots
\alpha_n$, where $\alpha_i$ is mapped to $T_i$ and its
endpoints are mapped to $x_{i-1},x_{i}$.
\end{defin}

\begin{rem}
\label{rem:diagram2} Any decomposition $\gamma_1\cdots
\gamma_n$ returning at $T_0$ bounds a disc diagram: It suffices
to consider arbitrary paths $\alpha_i$ embedded in $T_i$ joining $x_{i-1},x_{i}$
(modulo $n+1$) and a disc diagram for $\alpha_0\alpha_1\cdots
\alpha_n$.

Similarly, for any path $\alpha$ embedded in $X^{(1)}$ joining the
endpoints of a decomposition $\gamma_1\cdots
\gamma_n$, there is a disc diagram bounded by $\gamma_1\cdots
\gamma_n$ and $\alpha$.
\end{rem}

\begin{proof}[Proof of Proposition~\ref{prop:main}]
By Proposition~\ref{prop:shrink}(i) it suffices to show that
for any $n\leq N$ there is no tight decomposition $\gamma_1
\cdots \gamma_n$ returning at a tile $T_0 \in \mathcal{T}$.
Suppose that there is such a decomposition. By
Remark~\ref{rem:diagram2}, it is bounded by a diagram
$D\rightarrow X$. After passing to a subdiagram we can assume
that there is no $2$--cell in $D$ mapped to $T_i$ adjacent to
$\alpha_i$.

For every $T_{i+1}$ sharing a $2$--cell with $T_i$, replace the
tile $T_{i+1}$ with $T_{i+1}' = C_{i+1}$ and call it
\emph{shrunk}. Otherwise if $T_{i+1}$ shares no $2$-cells with
$T_i$ we define $T_{i+1}' = T_{i+1}$.

Let $Y\subset X$ be the subcomplex that is the union of $T_0,
T_1', \ldots, T_n'$ and the image of $D$. Let $\mathcal C$ be
the $2$--cells of $Y$ outside $T_0, T_1', \ldots, T_n'$. For
$i=1,\ldots, n$, let $P_i\subset T'_i$ be the span of the image
of $\alpha_i$.

For non-shrunk $T'_i=T_i$, by Remark~\ref{rem:smallantipodal}
we have $|P_i|\geq \mathrm{Bal}(T'_i)$. If $T'_i=C_i$ is
shrunk, the same is true except for the case where the edge
midpoints $x_{i-1},x_i\in C_i$ are at distance $<\frac{1}{2}l$.
In that case however, by
Proposition~\ref{prop:wallconstruction}(ii), the edge midpoint
$x_{i-1}'\in C_i$ antipodal to $x_{i}$ lies in $T_{i-1}$,
coinciding with $T'_{i-1}$ if $i>1$. We then append $P_i$ by an
edge-path joining $x_{i-1}'$ to $x_{i-1}$ in $\partial C_i\cap
T_{i-1}$, for which we keep the notation $P_i$ and which has
now at least $\frac{1}{2}l$ edges so that we trivially have
$|P_i|\geq \mathrm{Bal}(T'_i)$.

\begin{claim}
\label{cla:c=0}
We have $|\mathcal C|=0$ and $|Y|\leq 5$.
\end{claim}

\begin{proof}
We bound the cancellation in $Y$ from below using Remark~\ref{rem:cancel} with $\{Y_i\}=\{T_0\}\cup\{T_i'\}\cup\mathcal C$.

\begin{align*}
\mathrm{Cancel}(Y)&\geq \mathrm{Cancel}(T_0)+\sum_{i=1}^n\mathrm{Cancel}(T_i')+\frac{1}{2}\Big(\sum_{i=1}^n|P_i|+|\mathcal C|l\Big)\\
&\geq \mathrm{Cancel}(T_0)+
\sum_{i=1}^n\bigg(\frac{1}{2}\mathrm{Bal}(T'_i)+\mathrm{Cancel}(T'_i)\bigg)+\frac{1}{2}|\mathcal C|l\\
&=\mathrm{Cancel}(T_0)+
\sum_{i=1}^n\bigg(\frac{1}{8}(|T_i'|+1)l+\frac{1}{2}\mathrm{Cancel}(T'_i)\bigg)+\frac{1}{2}|\mathcal C|l\\
&\geq\frac{1}{4}(|T_0|-1)l+
\frac{1}{8}\sum_{i=1}^n\big((|T_i'|+1)l+(|T_i'|-1)l\big)+\frac{1}{2}|\mathcal C|l\\
&= \frac{1}{4}\big(|Y|-1+|\mathcal C|\big)l.
\end{align*}

Since $n\leq N$, and tiles of $\mathcal T$ have size $\leq 4$,
the quantity $|\partial D|/l$ is uniformly bounded. By
Theorem~\ref{thm:isoperimetry}, the size $|D|$ is uniformly
bounded, and so is $|Y|$ as well. By
Proposition~\ref{prop:isoperimetry}, we have $|\mathcal C|=0$,
and by the calculation in Remark~\ref{rem-tile-size-bound} we
have $|Y|\leq 5$, as desired.
\end{proof}

\begin{claim}
\label{cla:tree}
We have $|D|=0$, i.e.\ $D$ is a tree.
\end{claim}
\begin{proof}
Otherwise, consider a component $T$ in $D$ of the $2$--cells in
the preimage of some $T'_i$. In the calculation in the proof of
Claim~\ref{cla:c=0} we can thus replace $P_i$ with the image of
$\partial T$ in $Y$. By Corollary~\ref{cor:no_short_image} we
have now $|P_i|\geq l$. Thus
$\mathrm{Cancel}(T'_i)+\frac{1}{2}|P_i|\geq
\frac{1}{4}(|T'_i|+1)l$. On the other hand, the term
$\mathrm{Cancel}(T'_i)+\frac{1}{2}|P_i|$ was estimated in the
proof of Claim~\ref{cla:c=0} only by $\frac{1}{4}|T'_i|l$,
which gives an extra $\frac{1}{4}l$ which violates
Proposition~\ref{prop:isoperimetry}.
\end{proof}

Note that $n>1$ by Lemma~\ref{lem:walls_do_not_backtrack} (and Proposition~\ref{prop:wallconstruction}(ii) if $T_1'$ is shrunk).

\begin{claim}
\label{cla:p>2}
We have $n>2$.
\end{claim}
\begin{proof}
Otherwise, by Claim~\ref{cla:tree}, the disc diagram $D$ is a
tripod. If $|T'_1\cap T'_2|\geq \frac{1}{4}l$, then since
$|T'_1|+|T'_2|\leq 4$, the tiles $T'_1$ and $T'_2$ (one of
which could be shrunk) would have been glued into one tile in
Construction~\ref{constr:main}, which is a contradiction.
Otherwise,
\begin{align*}
\mathrm{Cancel}(Y)&\geq \mathrm{Cancel}(T_0)+\sum_{i=1}^2 \mathrm{Cancel}(T_i')+\sum_{i=1}^2 |P_i|-\frac{1}{4}l\\
&\geq\frac{1}{4}(|T_0|-1)l +\sum_{i=1}^2 \frac{1}{4}(|T_i'|+1)l-\frac{1}{4}l=\frac{1}{4}|Y|l,
\end{align*}
which contradicts Proposition~\ref{prop:isoperimetry}.
\end{proof}

\begin{claim}
\label{cla:consecutive_tiles} There is $1\leq i\leq n$ and
$j=i\pm 1$ (modulo $n+1$) such that $T'_i\cup T'_j$ is a tile.
\end{claim}

\begin{proof}
By Claim~\ref{cla:c=0} we have $|Y|\leq 5$ and by
Claim~\ref{cla:p>2} we have $n\geq 3$, so that the values of
$|T_0|, |T'_i|$ are $1$ or $2$. In particular, none of $T_i'$
is shrunk and we do not append $P_i$. By Claim~\ref{cla:tree},
the disc diagram $D$ is a tree. Thus there is $1\leq i\leq n$
such that $\alpha_i$ is contained in $\alpha_{i-1}\cup
\alpha_{i+1}$. We choose $j\in \{i-1,i+1\}$ with larger
$|P_i\cap P_j|$. By Remark~\ref{rem:smallantipodal}, we have
$|P_i\cap P_{j}|\geq \frac{1}{2}\mathrm{Bal}(T'_i)$.

To prove that $T'_i\cup T'_j$ is a tile we compute its cancellation.
\begin{align*}
\mathrm{Cancel}(T_{i}'\cup T'_{j})&\geq \mathrm{Cancel}(T'_{j})+\frac{1}{2}\mathrm{Bal}(T'_i)+\mathrm{Cancel}(T'_{i})\\
&\geq\frac{1}{4}(|T'_{j}|-1)l+\frac{1}{8}(|T'_{i}|+1)l+\frac{1}{2}\mathrm{Cancel}(T'_{i})\\
&\geq\frac{1}{4}(|T'_{j}|-1)l+\frac{1}{4}|T'_{i}|l\qedhere
\end{align*}
\end{proof}

To obtain the final contradiction, it suffices to observe that
Claim~\ref{cla:consecutive_tiles} saying that $T'_i\cup T'_j$
is a tile, whereas $|T'_i|+|T'_j|\leq 3$, contradicts
Construction~\ref{constr:main}.
\end{proof}

\section{Quasi-convexity}
\label{sec:quasi-convexity}

In this section we complete the proof of
Theorem~\ref{thm:walls_trees} using
Theorem~\ref{thm:quasiisometry} saying that hypergraphs are
quasi-isometrically embedded in $X$. It is intriguing that we
will not use Theorem~\ref{thm:quasiisometry} directly in the
proof of Theorem~\ref{thm:main} but it is difficult to
circumvent to obtain Theorem~\ref{thm:walls_trees}.

Let $\gamma$ be a hypergraph segment. Consider the path metric
on $\gamma$ where all diagonals have length $\frac{1}{2}l$. Let
$\mathfrak{G}$ be the vertex set of $\gamma$ with the
restricted metric.

\begin{thm}
\label{thm:quasiisometry} Let $d<\frac{5}{24}$. There are
constants $\Lambda, c$, such that w.o.p.\ the map from the
vertex set $\mathfrak{G}$ of any hypergraph segment to
$X^{(1)}$ is a $(\Lambda, cl)$--quasi-isometric embedding.
\end{thm}

Before we prove Theorem~\ref{thm:quasiisometry}, we give the following.

\begin{proof}[Proof of Theorem~\ref{thm:walls_trees}]
Suppose that we have a hypergraph segment $xx'$ that
self-intersects, i.e.\ starts and ends at the same $2$--cell,
so that $|x,x'|_{X^{(1)}}\leq \frac{1}{2}l$. Denote by $n$ the
number of $2$--cells traversed by $xx'$. By
Theorem~\ref{thm:quasiisometry}, we have
$$|x,x'|_{X^{(1)}}\geq \frac{1}{\Lambda}\bigg(n\frac{1}{2}l\bigg)-cl.$$
Hence we have an a priori bound $n\leq (2c+1)\Lambda$. Thus by
Remark~\ref{rem:diagram} it suffices to apply
Proposition~\ref{prop:main} with $N=(2c+1)\Lambda$.
\end{proof}

\begin{proof}[Proof of Theorem~\ref{thm:quasiisometry}]
The Cayley graph $X^{(1)}$ of a random group at a fixed density
$d<\frac{1}{2}$ is w.o.p.\ hyperbolic with the hyperbolicity
constant a linear function of $l$. We thus rescale the metric
on $X^{(1)}$ by $\frac{1}{l}$ so that the constant is uniform.
We also rescale by $\frac{1}{l}$ the metric on $\gamma\supset
\mathfrak{G}$. We appeal to \cite[Thm 5.21]{GH} implying that
local quasigeodesics are quasigeodesics. More precisely, to
prove Theorem~\ref{thm:quasiisometry}, it suffices to find
$\lambda$ such that for some sufficiently large $N=N(\lambda)$ the
map to $X^{(1)}$ from any $\mathfrak{G}$ of cardinality $\leq
N$ is $\lambda$--bilipschitz. We will do that for
$\lambda=\frac{1}{1-4d}$.

Let $\gamma$ be a hypergraph segment with vertex set
$\mathfrak{G}$ of cardinality $\leq N$ and let $\alpha$ be a
geodesic in $X^{(1)}$ joining the endpoints of $\gamma$.

By taking for $C_i$ consecutive $2$--cells traversed by
$\gamma$, we see that $\gamma$ has a decomposition $\gamma_1
\cdots \gamma_n$ and by Propositions~\ref{prop:shrink}(ii)
and~\ref{prop:main} we can assume that it is tight. We put
$T'_1=T_1$ and as in the proof of Proposition~\ref{prop:main},
for every $T_{i+1}$ sharing a $2$--cell with $T_i$ we define
$T_{i+1}' = C_{i+1}$ and call it \emph{shrunk}, and otherwise
we take $T_{i+1}' = T_{i+1}$. By Remark~\ref{rem:diagram2},
there is a disc diagram $D\rightarrow X$ bounded by $\gamma_1
\cdots \gamma_n$ and $\alpha$. After passing to a subdiagram we
can assume that there is no $2$--cell in $D$ mapped to $T_i$
adjacent to $\alpha_i$.

Let $Y\subset X$ be the subcomplex that is the union of $T_i'$
and the image of $D$. Let $\mathcal{C}$ be the $2$--cells of
$Y$ outside $\bigcup_i T_i'$. Let $P_i\subset T'_i$ be the span
of the image of $\alpha_i$, which we append as in the proof of
Proposition~\ref{prop:main} for $T_i'$ shrunk. Thus $|P_i|\geq
\mathrm{Bal}(T'_i)$. We estimate the cancellation in $Y$ using
Remark~\ref{rem:cancel}. Recall that we rescaled the length of
$\alpha$ by factor $l$.
\begin{align*}
\mathrm{Cancel}(Y)&\geq \sum_{i=1}^n\mathrm{Cancel}(T_i')+\frac{1}{2}\Big(\sum_{i=1}^n|P_i|+|\mathcal C|l-|\alpha|l\Big)\\
&\geq
\sum_{i=1}^n\bigg(\frac{1}{2}\mathrm{Bal}(T'_i)+\mathrm{Cancel}(T'_i)\bigg)+\frac{1}{2}|\mathcal C|l-\frac{1}{2}|\alpha|l\\
&\geq\frac{1}{4}|Y|l-\frac{1}{2}|\alpha|l .
\end{align*}

Note that the value $N=N(\lambda)$ gives a uniform bound on the
length of $\alpha$ and consequently by
Theorem~\ref{thm:isoperimetry} a uniform bound on the size of
$D$ and $Y$. We can thus apply Proposition~\ref{prop:isoperimetry} to $Y$, which yields
$$2\bigg(\frac{1}{4}-d\bigg)|Y|\leq |\alpha|.$$
The distance $|\gamma|$ between the endpoints of $\gamma$ in
$\mathfrak{G}$ is $\leq\frac{1}{2}|Y|$, thus we obtain
$\frac{1}{\lambda}|\gamma|\leq|\alpha|$. Since consecutive
points of $\mathfrak{G}$ are at distance $\leq \frac{1}{2}$ in
$X^{(1)}$, we also have $|\alpha|\leq |\gamma|$. Thus
$\mathfrak{G}\rightarrow X^{(1)}$ is $\lambda$--bilipschitz, as
desired.
\end{proof}

Before we prove Theorem~\ref{thm:main}, we need the following lemma.

\begin{lem}
\label{lem:transitive}
Let $d<\frac{5}{24}$. W.o.p.\ there is a hypergraph in $X$ and an element of its stabiliser in $G$ exchanging its two complementary components.
\end{lem}

Actually $G$ acts transitively on the set of complementary components of all hypergraphs, but we do not need this.

\begin{proof}
We will prove that there is a relator $r\in R$ such that
\begin{enumerate}[(1)]
\item
there are two antipodal occurrences of a letter $s$ in $r$, and
\item
the relation $\sim_C$ on some (hence any) $2$--cell $C$ corresponding to $r$ is antipodal.
\end{enumerate}
This suffices to prove the lemma, for if a hypergraph contains a diagonal $\gamma$ connecting the midpoints of the directed edges $e,e'$ of~$C$
labelled by same letter, then there exists $g \in G$ with $ge = e'$, and so $g$ stabilises
that hypergraph. Since inside $C$ the edges $e,e'$ cross $\gamma$ in opposite directions, we see that $g$
exchanges the complementary components of this hypergraph.

We claim that w.o.p.\ the first relator $r_1\in R$ satisfies condition (1).
Since there are $2m$ letters in $S\cup S^{-1}$, the probability
that a fixed antipodal edge pair is labelled by the same letter
is nearly $\frac{1}{2m}$ as $l\rightarrow\infty$, and
conditioned on the event that a preceding antipodal edge pair
is not labelled by the same letter is nearly
$\frac{2m-2}{(2m-1)^2}$. However, there are $\frac{1}{2}l$
antipodal edge pairs in $r_1$ so the probability that none of
them is labelled by the same letter tends to $0$ as
$l\rightarrow\infty$, justifying the claim.

By Construction~\ref{constr:main}, the map from a tile in $\mathcal{T}$
to $X/G$ maps distinct $2$--cells to distinct $2$--cells. In particular, if a $2$--cell of $X$ corresponding to $r_1$ lay in a tile of size $>1$,
then this would contradict Corollary~\ref{cor:rel_notintile}. Thus condition (2) is satisfied as well.
\end{proof}

\begin{proof}[Proof of Theorem~\ref{thm:main}]
Let $F\subset G$ be the stabiliser of a hypergraph $\Gamma$ in
$X$ satisfying Lemma~\ref{lem:transitive}. Note that $F$ acts
on $\Gamma$ cocompactly since $\sim_C$ are $G$--invariant.

We claim that the components of $X-\Gamma$ are
\emph{essential}, i.e.\ they are not at finite distance from
$\Gamma$. Otherwise, since they are exchanged by $F$, both
components of $X-\Gamma$ are at finite distance from $\Gamma$.
Thus $F$ acts cocompactly on $X$ and consequently $G$ is
quasi-isometric to $F$ hence to $\Gamma$, which is a
tree by Theorem~\ref{thm:walls_trees}. Recall that $G$ is
w.o.p.\ torsion free, since $X$ is contractible
\cite{Gro,O-hyp}. Thus by Stallings Theorem \cite{St}, the
group $G$ is free and hence $\chi(G)\leq 0$. But on the other
hand we have $\chi(G)=1-m+\lfloor(2m-1)^{dl}\rfloor>0$ for $l$
large enough, which is a contradiction.

This justifies the claim that the components of $X-\Gamma$ are
essential. Let $F'\subset F$ be the index $2$ subgroup
preserving the components of $X-\Gamma$. Then the number of
relative ends $e(G,F')$ is greater than $1$. Thus Sageev's
construction (\cite[Thm~3.1]{S}, see also \cite{Ger} and
\cite{NR}) gives rise to a nontrivial action of $G$ on a CAT(0)
cube complex. By \cite{NR} the group $G$ does not satisfy
Kazhdan's property~$\mathrm{(T)}$.
\end{proof}

\begin{rem}
\label{rem:finite dim} By Theorem~\ref{thm:quasiisometry}, the
subgroup $F'$ above is quasi-isometrically embedded in the
hyperbolic group $G$. Thus by \cite[Thm~3.1]{S2} or~\cite{GMRS}
there is a CAT(0) cube complex satisfying
Theorem~\ref{thm:main} for which the action of $G$ is cocompact
and the complex is finite dimensional.
\end{rem}

\end{document}